\documentclass{amsart}
\usepackage{amsfonts,amsmath,amsthm,amssymb,color}
\usepackage[all]{xy}
\usepackage[pagebackref,colorlinks=true,linkcolor=blue,urlcolor=blue,citecolor=blue]{hyperref}
\usepackage{eucal,comment}
\numberwithin{equation}{section}
\providecommand{\eprint}[2][]{\href{http://arxiv.org/abs/#2}{arXiv:#2}}

\begin{document}

\newcommand{\QUI}{\bigskip ****** \textbf{arrivato qui!} *******\bigskip}

\newtheorem{theorem}{Theorem}[section]
\newtheorem{thm}[theorem]{Theorem}
\newtheorem{lemma}[theorem]{Lemma}
\newtheorem{proposition}[theorem]{Proposition}
\newtheorem{corollary}[theorem]{Corollary}

\theoremstyle{definition}
\newtheorem{definition}[theorem]{Definition}
\newtheorem{example}[theorem]{Example}

\theoremstyle{remark}
\newtheorem{remark}[theorem]{Remark}

\newenvironment{magarray}[1]
{\renewcommand\arraystretch{#1}}
{\renewcommand\arraystretch{1}}

\newcommand{\quot}[2]{
{\lower-.2ex \hbox{$#1$}}{\kern -0.2ex /}
{\kern -0.5ex \lower.6ex\hbox{$#2$}}}

\newcommand{\mapor}[1]{\smash{\mathop{\longrightarrow}\limits^{#1}}}
\newcommand{\mapin}[1]{\smash{\mathop{\hookrightarrow}\limits^{#1}}}
\newcommand{\mapver}[1]{\Big\downarrow
\rlap{$\vcenter{\hbox{$\scriptstyle#1$}}$}}
\newcommand{\liminv}{\smash{\mathop{\lim}\limits_{\leftarrow}\,}}

\newcommand{\Set}{\mathbf{Set}}
\newcommand{\Art}{\mathbf{Art}}
\newcommand{\solose}{\Rightarrow}

\newcommand{\specif}[2]{\left\{#1\,\left|\, #2\right. \,\right\}}

\renewcommand{\bar}{\overline}
\newcommand{\de}{\partial}
\newcommand{\debar}{{\overline{\partial}}}
\newcommand{\per}{\!\cdot\!}
\newcommand{\Oh}{\mathcal{O}}
\newcommand{\sA}{\mathcal{A}}
\newcommand{\sB}{\mathcal{B}}
\newcommand{\sC}{\mathcal{C}}
\newcommand{\sD}{\mathcal{D}}
\newcommand{\sE}{\mathcal{E}}
\newcommand{\sF}{\mathcal{F}}
\newcommand{\sG}{\mathcal{G}}
\newcommand{\sI}{\mathcal{I}}
\newcommand{\sH}{\mathcal{H}}
\newcommand{\sK}{\mathcal{K}}
\newcommand{\sL}{\mathcal{L}}
\newcommand{\sM}{\mathcal{M}}
\newcommand{\sP}{\mathcal{P}}
\newcommand{\sU}{\mathcal{U}}
\newcommand{\sV}{\mathcal{V}}
\newcommand{\sX}{\mathcal{X}}
\newcommand{\sY}{\mathcal{Y}}
\newcommand{\sN}{\mathcal{N}}
\newcommand{\sT}{\mathcal{T}}
\newcommand{\Aut}{\operatorname{Aut}}
\newcommand{\Id}{\operatorname{Id}}
\newcommand{\Tr}{\operatorname{Tr}}
\newcommand{\Mor}{\operatorname{Mor}}
\newcommand{\Def}{\operatorname{Def}}
\newcommand{\Fitt}{\operatorname{Fitt}}
\newcommand{\Hom}{\operatorname{Hom}}
\newcommand{\Hilb}{\operatorname{Hilb}}
\newcommand{\HOM}{\operatorname{\mathcal H}\!\!om}
\newcommand{\EXT}{\operatorname{\mathcal E}\!\!xt}
\newcommand{\DER}{\operatorname{\mathcal D}\!er}
\newcommand{\Spec}{\operatorname{Spec}}
\newcommand{\Der}{\operatorname{Der}}
\newcommand{\Tor}{{\operatorname{Tor}}}
\newcommand{\Ext}{{\operatorname{Ext}}}
\newcommand{\End}{{\operatorname{End}}}
\newcommand{\END}{\operatorname{\mathcal E}\!\!nd}
\newcommand{\Image}{\operatorname{Im}}
\newcommand{\coker}{\operatorname{coker}}
\newcommand{\tot}{\operatorname{tot}}
\newcommand{\ten}{\bigotimes}
\newcommand{\mA}{\mathfrak{m}_{A}}

\newcommand{\somdir}[2]{\hbox{$\mathrel
{\smash{\mathop{\mathop \bigoplus\limits_{#1}}
\limits^{#2}}}$}}
\newcommand{\tensor}[2]{\hbox{$\mathrel
{\smash{\mathop{\mathop \bigotimes\limits_{#1}}
^{#2}}}$}}
\newcommand{\symm}[2]{\hbox{$\mathrel
{\smash{\mathop{\mathop \bigodot\limits_{#1}}
^{#2}}}$}}
\newcommand{\external}[2]{\hbox{$\mathrel
{\smash{\mathop{\mathop \bigwedge\limits_{#1}}
^{\!#2}}}$}}

\renewcommand{\Hat}[1]{\widehat{#1}}
\newcommand{\dual}{^{\vee}}
\newcommand{\desude}[2]{\dfrac{\de #1}{\de #2}}

\newcommand{\A}{\mathbb{A}}
\newcommand{\N}{\mathbb{N}}
\newcommand{\R}{\mathbb{R}}
\newcommand{\Z}{\mathbb{Z}}
\renewcommand{\H}{\mathbb{H}}
\renewcommand{\L}{\mathbb{L}}
\newcommand{\proj}{\mathbb{P}}
\newcommand{\K}{\mathbb{K}\,}
\newcommand\C{\mathbb{C}}
\newcommand\Del{\operatorname{Del}}
\newcommand\Tot{\operatorname{Tot}}
\newcommand\Grpd{\mbox{\bf Grpd}}

\newcommand{\g}{\mathfrak{g}}

\newcommand\é{\'e}
\newcommand\è{\`e}
\newcommand\à{\`a}
\newcommand\ì{\`i}
\newcommand\ù{\`u}
\newcommand\ò{\`o }


\newcommand{\rh}{\rightarrow}
\newcommand{\contr}{{\mspace{1mu}\lrcorner\mspace{1.5mu}}}

\newcommand{\bi}{\boldsymbol{i}}
\newcommand{\bl}{\boldsymbol{l}}

\newcommand{\MC}{\operatorname{MC}}
\newcommand{\TW}{\operatorname{TW}}

\newenvironment{acknowledgement}{\par\addvspace{17pt}\small\rm
\trivlist\item[\hskip\labelsep{\it Acknowledgement.}]}
{\endtrivlist\addvspace{6pt}}


\title[Endomorphisms of Koszul complexes]{Endomorphisms of Koszul complexes: formality and application to deformation theory}
\date{\today}

\author{Francesca Carocci}
\address{\newline
University of Edinburgh,\hfill\newline
School of Mathematics \lq\lq James Clerk Maxwell building\rq\rq,\hfill\newline
    The King's Buildings,
    EH9 3FD, United Kingdom.}
\email{francesca.carocci@ed.ac.uk}

\author{Marco Manetti}
\address{\newline
Universit\`a degli studi di Roma ``La Sapienza'',\hfill\newline
Dipartimento di Matematica \lq\lq Guido
Castelnuovo\rq\rq,\hfill\newline
P.le Aldo Moro 5,
I-00185 Roma, Italy.}
\email{manetti@mat.uniroma1.it}
\urladdr{\href{http://www.mat.uniroma1.it/people/manetti/}{www.mat.uniroma1.it/people/manetti/}}

\begin{abstract} We study the differential graded Lie algebra of endomorphisms of the Koszul resolution 
of a regular sequence on a unitary commutative $\K$-algebra $R$ and we prove that it is homotopy abelian over $\K$ but not over $R$ (except trivial cases). We apply this result to prove an annihilation theorem 
for obstructions of (derived) deformations of locally complete intersection ideal sheaves on projective schemes.
\end{abstract}

\subjclass[2010]{17B70,14D15, 18G50, 13D10}
\keywords{Koszul complex, Deformations of  coherent sheaves, differential graded Lie algebras}

\maketitle

\section*{Introduction}
Let $R$ be a fixed unitary commutative ring, then the usual notion of differential graded (DG) Lie algebra over a field  extends naturally to the  notion of  DG-Lie algebra over $R$ (Definition~\ref{def.dgla}). Recently DG-Lie algebras over commutative rings have received some attention  in the study of degenerations of Batalin-Vilkoviski algebras \cite{KKP} and also in the study of formality in families \cite{kaledin,lunts,dglaformality}. Also the notions of quasi-isomorphism, formality and homotopy abelianity (i.e., formality plus  the bracket trivial in cohomology) extend without difficulty to DG-Lie algebras over $R$, but the analysis of examples shows immediately that these concepts become very restrictive when working over a ring, and should be replaced with more convenient notions. 

In this paper we restrict our attention to homotopy abelianity and we introduce the notion of \emph{numerically homotopy abelian} (NHA)  DG-Lie algebra (Definition~\ref{def.nhodgla}) which seems more useful, at least for the application in deformation theory. 
Very briefly, the class of NHA DG-Lie algebras is the smallest class satisfying the following conditions:
\begin{enumerate} 
\item if $L$ is abelian, then   $L$ is NHA;
\item if $L\to M$ is a morphism surjective in cohomology and $L$ is NHA, then also  $M$ is NHA;
\item if $L\to M$ is a morphism injective in cohomology and $M$ is NHA, then also  $L$ is NHA.
\end{enumerate}

It is plain that every homotopy  abelian  DG-Lie algebra is also numerically homotopy abelian, and  every numerically homotopy abelian algebra has  trivial bracket in cohomology. It is well known that over a field of characteristic 0  every 
numerically homotopy abelian DG-Lie algebra is also homotopy abelian and that there exist DG-Lie algebras with trivial bracket 
in cohomology that are not homotopy abelian.

The first goal of this paper is to provide some examples, which occur in deformation theory of locally complete intersections, 
of numerically homotopy abelian DG-Lie algebras that are not homotopy abelian: more precisely we 
shall prove that if $f_1,\ldots,f_r\in R$ is a regular sequence contained in a proper ideal and $K^*=K(f_1,\ldots,f_r)$ is its Koszul complex,  then the DG-Lie algebras $\Hom_R^*(K^*,K^*)$ 
of $R$-linear endomorphisms of $K^*$ is  not homotopy abelian over $R$,

On the other side,  we shall prove that the DG-Lie algebra $L$ of graded endomorphisms for the Koszul complex $K^*$ (or its truncated $K^{<0}$) for a regular sequence in $R$ is numerically homotopy abelian (Theorems~\ref{thm.nhakoszul} and \ref{thm.nhakoszultronco}). This implies that if $R$ is a $\K$-algebra, with $\K$ field of characteristic 0, then $L$ is homotopy abelian over $\K$ and therefore the associated 
Maurer-Cartan/gauge functor is unobstructed.

Clearly this fact has a number of consequences in deformation theory of locally complete intersections. For instance   
in Section~\ref{sec.deformations} we shall prove the following  result:

\begin{theorem}
Let $\sI\subset \Oh_X$ be a locally complete intersection ideal sheaf on a projective scheme $X$ over a field $\K$ of characteristic 0, and denote by $\Oh_Z=\Oh_X/\sI$ the structure sheaf on the closed subscheme $Z$ defined by $\sI$.
Let $\sF$ be either a line bundle over $Z$ or $\sF=\sI\otimes \sL$ for a line bundle $\sL$ on $X$.
Then:\begin{enumerate}

\item  the obstructions to deforming $\sF$ are contained in the kernel of the natural map
\[ \alpha_0\colon \Ext^2_{X}(\sF,\sF)\to H^0(X,\EXT^2_{X}(\sF,\sF))\,.\]   
In particular, if $\alpha_0$ is injective then $\sF$ has unobstructed deformations;

\item the obstructions to derived deformations of $\sF$ are contained in the kernel of the natural map of graded vector spaces 
\[ \alpha\colon \bigoplus_n\Ext^{n+2}_{X}(\sF,\sF)\to \bigoplus_n H^0(X,\EXT^{n+2}_{X}(\sF,\sF))\,.\]   
In particular, if $\alpha$ is injective then the DG-Lie algebra $\operatorname{Tot}(L(\sU,\sE^{\bullet}))$ controlling deformations of $\sF$ is homotopy abelian over $\K$.
\end{enumerate}
\end{theorem}

The underlying idea is that, since the DG-Lie algebra of endomorphisms of a Koszul resolution is a homotopy abelian over $\K$, 
the restriction to an affine open set of a sheaf $\sF$ as above has unobstructed (derived) deformations; it is then sufficient to 
consider the morphism from global to local deformations of $\sF$. 
To ensure that the passage from global to local is a ``genuine'' morphism of deformation theories, we shall see that 
the maps $\alpha_0,\alpha$ are induced in cohomology by a morphism of DG-Lie algebras, namely from the DG-Lie algebra controlling global deformations to the one controlling local deformation. This is done by fixing  a locally free resolution $\sE^{\bullet}$ of $\sF$ (hence the assumption $X$ projective) to then exploit a recent result by Fiorenza-Iacono-Martinengo \cite{FIM}
proving that local and global deformations of $\sF$ are controlled by the quasi-coherent sheaf of DG-Lie algebras $\sH om^*_X(\sE^{\bullet},\sE^{\bullet})$.
 
\begin{remark} After the first version of this paper was posted on arXiv,  F. Meazzini proposed  a new approach to deformations of coherent sheaves via local resolution \cite{meazzini} that does not require 
the existence of locally free resolutions and that allows to replace in the above theorem the assumption $X$ projective with $X$ separated and noetherian scheme of finite dimension: however this weakening of assumption is technically quite complicated and for simplicity of exposition we maintain here the assumption $X$ projective.  
\end{remark}

\bigskip
\section{DG-Lie algebras over commutative rings}

In what follows $R$ will be a unitary commutative ring. 

\begin{definition}\label{def.dgla} 
A \textbf{differential graded Lie algebra} (DG-Lie algebra) over $R$ is the data of cochain complex of $R$-modules  
$(L,d)$ equipped with an $R$-bilinear bracket
$[-,-]\colon L\times L\to L$  satisfying 
the following conditions:\begin{enumerate}
\item $[-,-]$ is homogeneous graded skewsymmetric.  
This means that:\begin{enumerate}
\item  $[L^i,L^j]\subset L^{i+j}$,

\item  $[a,b]+(-1)^{\bar{a}\bar{b}}[b,a]=0$ for every $a,b$ homogeneous,

\item $[a,a]=0$ for every homogeneous $a$ of even degree;
\end{enumerate} 

\item (Leibniz identity) 
$d[a,b]=[da,b]+(-1)^{\bar{a}}[a,db]$; 

\item (Jacobi identity) every triple of homogeneous elements 
$a,b,c$  satisfies the equality  
\[ [a,[b,c]]=[[a,b],c]+(-1)^{\bar{a}\;\bar{b}}[b,[a,c]]\,;\]

\item  (Bianchi identity) $[b,[b,b]]=0$ for  every homogeneous $b$ of odd degree.

\end{enumerate}
Morphisms of DG-Lie algebras are morphisms of cochain complexes of $R$-modules commuting with brackets. The resulting category is denoted by $\mathbf{DGLA}(R)$.
\end{definition}

Notice that the Jacobi identity implies the Bianchi identity whenever $3$ is not a zero-divisor  in $R$. A quasi-isomorphism of DG-Lie algebras is a morphism of DG-Lie algebras which is also a quasi-isomorphism of complexes. A DG-Lie algebra with trivial bracket is called abelian. 

\begin{definition}\label{def.hadgla} A DG-Lie algebra $L$ is said to be \emph{homotopy abelian} if there exists a finite zigzag of morphisms of DG-Lie algebras 
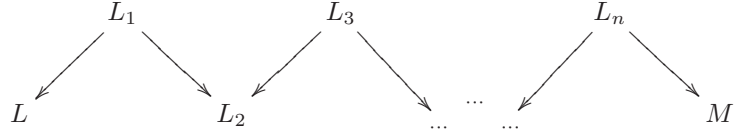
\[ \xymatrix{&L_1\ar[dl]\ar[dr]&&L_3\ar[dl]\ar[dr]&&&L_n\ar[dl]\ar[dr]&\\
L&&L_2&&{}\ar@{}[r]_{\cdots\qquad\cdots}^{\cdots}&{}&&M}\]
such that  every arrow is a quasi-isomorphism and $M$ is  abelian.
\end{definition}

 As we said in the introduction, over a general ring the notion of homotopy abelianity seems to be too restrictive and it is convenient to consider a suitable more general class.

\begin{definition}\label{def.nhodgla} A DG-Lie algebra $L$ is said to be \emph{numerically homotopy abelian} if there exists a finite zigzag of morphisms of DG-Lie algebras 
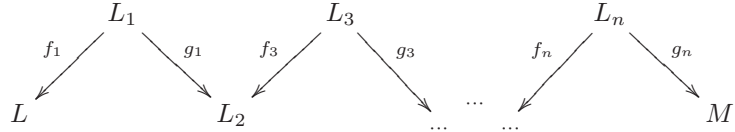
\[ \xymatrix{&L_1\ar[dl]_{f_1}\ar[dr]^{g_1}&&L_3\ar[dl]_{f_3}\ar[dr]^{g_3}&&&L_n\ar[dl]_{f_n}\ar[dr]^{g_n}&\\
L&&L_2&&{}\ar@{}[r]_{\cdots\qquad\cdots}^{\cdots}&{}&&M}\]
such that: 
\begin{itemize} 
\item the DG-Lie algebra $M$ is abelian,

\item every morphism $f_i$ is surjective in cohomology,

\item every morphism $g_i$ is injective in cohomology.

\end{itemize}

\end{definition}

It is plain that every homotopy abelian DG-Lie algebra is also numerically homotopy abelian, and that for every numerically homotopy abelian DG-Lie algebra the bracket is trivial in cohomology. 

\begin{remark}\label{rem.nha}
It is well known (see e.g. \cite[Lemma 6.1]{FMpoisson}, \cite[Lemma 2.11]{iacono2017} , \cite[Prop. 4.11]{KKP}) that if $R$ is a field of characteristic 0, then every numerically homotopy abelian DG-Lie algebra is homotopy abelian. 
For instance if $M$ is an abelian Lie algebra and $f\colon M\to L$ is a morphism of DG-Lie algebras that is surjective in cohomology, when $R$ is a field it is  sufficient to consider the restriction of $f$ to a suitable subcomplex $N\subset M$ such that $f\colon N\to L$ is a quasi-isomorphism. The proof of the general case follows the same ideas but requires the factorisation lemma and characteristic $0$.  
\end{remark}

One of the goals of this paper is to show that over general commutative rings the class of numerically homotopy abelian DG-Lie algebras strictly contains 
the class of homotopy abelian DG-Lie algebra. More precisely,  we shall prove that the algebra of endomorphisms of the Koszul complex of a (non trivial) regular sequence is  numerically homotopy abelian but not homotopy abelian.
A simple obstruction to  homotopy abelianity is given by the following result, in which the symbol $\otimes_R^L$ denotes  the usual derived tensor product of complexes of $R$-modules.

\begin{lemma}\label{lem.criterioabelianita} 
Let $H$ be a homotopy abelian DG-Lie algebra over a ring $R$. Then the class of the bracket is trivial in the group
$\Ext^0_R(H\otimes^L_RH,H)$.
\end{lemma}

\begin{proof} The result is clearly  true whenever the bracket of $H$ is trivial. 
If $H\to K$ is a quasi-isomorphism of DG-Lie algebras, the commutative diagram 
\[ \xymatrix{H\otimes_R H\ar[r]^-{[-,-]}\ar[d]&H\ar[d]\\
K\otimes_R K\ar[r]^-{[-,-]}&K}\]
gives a commutative diagram in the homotopy category of complexes of $R$-modules
\[ \xymatrix{H\otimes_R^L H\ar[r]^-{[-,-]}\ar[d]&H\ar[d]\\
K\otimes_R^L K\ar[r]^-{[-,-]}&K}\]
with the vertical arrows isomomorphisms. Thus the upper horizontal arrow is homotopically 
trivial if and only if the lower  horizontal arrow is.

For later use it is useful to  give also a longer but more constructive proof of the lemma: 
every quasi-isomorphism of DG-Lie algebras $H\to K$ lifts to a morphism of complexes 
between cofibrant resolutions 
\[ \xymatrix{P\ar[r]\ar[d]^{\eta}&H\ar[d]\\
Q\ar[r]&K}\]
where cofibrant is intended with respect to the projective model structure \cite[Thm. 2.3.11]{hovey}. 
This gives a commutative diagram of complexes  
\[ \xymatrix{P\otimes_R P\ar[r]^-{[-,-]}\ar[d]_{\eta\otimes \eta}&H\ar[d]\\
Q\otimes_R Q\ar[r]^-{[-,-]}&K}\]
with the vertical arrows quasi-isomorphisms. Since $P\otimes_R P$ and $Q\otimes_R Q$ are cofibrant complexes the above diagram induces a pair of quasi-isomorphisms 
\[ \Hom^*_R(P\otimes_R P,H)\xrightarrow{\quad} \Hom^*_R(P\otimes_R P,K)
\xleftarrow{\quad}\Hom_R^*(Q\otimes_R Q,K)\]
and then the class of the bracket of $H$ is trivial in 
$ \Ext^0_R(H\otimes^L_RH,H)=H^0(\Hom_R^*(P\otimes_R P,H))$ if and only if the same holds for $K$.
\end{proof}

The following example remarks that we really need to consider the class of the bracket in the derived 
category. Suppose $L$ is homotopy abelian over $R$, but not cofibrant \emph{as a complex}, i.e. with respect to the projective module structure on complexes of $R$-modules \cite{hovey}. In the latter case we can have that $[-,-]\in H^0(\Hom^*_R(L\otimes_R L, L))$ is not homotopic to zero.

\begin{example} 

Given an integral domain $R$ and a non invertible element $t\in R,$  denote 
$A=R/(t)$ and consider the DG-Lie algebra over $R$:
\[ L\colon\quad Rx\xrightarrow{\;d\;}Ry \xrightarrow{\;d\;}Az,\qquad \deg(x)=0,\quad \deg(y)=1,\quad\deg(z)=2,\]
\[ dx=ty,\quad dy=z,\quad [x,y]=y,\quad [x,z]=z,\quad [y,y]=0\;.\]
The verification of Leibniz, Bianchi and Jacobi is completely straightforward.  
Although it is acyclic, and then homotopy abelian, the bracket 
\[[-,-]\colon L\otimes_R L\to L\]
is not homotopic to 0 and then its class is not trivial in $H^0(\Hom^*_R(L\otimes_R L, L))$.
In fact, if  the bracket is the coboundary of a $R$-linear map 
$q\colon L\otimes_R L\to L$ of degree $-1$, 
then $q(x,z)=0$ since 
$tq(x,z)=q(x,tz)=0$, and the relation
\[ y=[x,y]=dq(x,y)+q(ty,y)+q(x,z)\in tRy\]
gives a contradiction.
\end{example}

\begin{corollary}\label{cor.nonabelianityresolution} 
Let $R$ be a commutative unitary ring and let 
\[F^*\colon\qquad F^{-n}\xrightarrow{\,\debar\,}\cdots\xrightarrow{\,\debar\,}F^{-1}
\xrightarrow{\,\debar\,}F^0\]
be a complex  of length $n>0$ of finitely generated free $R$-modules. Assume that there exists a proper ideal $I\subset R$ such that $\debar(F^i)\subset IF^{i+1}$ for every $i$. Then the DG-Lie algebra $H=\Hom^*_R(F^*,F^*)$ is not homotopy abelian over $R$.
\end{corollary}

\begin{proof} By the assumption on the length of the complex we have $F^0,F^{-n}\not=0$ and we can choose a morphism of $R$-modules $\beta\colon F^{-n}\to F^0$ such that 
$\beta(F^{-n})\not\subset IF^0$.   
Since  $H=\Hom^*_R(F^*,F^*)$  is a bounded complex of free $R$-modules, we have 
\[ \Ext^0_R(H\otimes^L_RH,H)=H^0(\Hom_R^*(H\otimes_RH,H))\,.\]
Moreover, 
$d(H^i)\subset IH^{i+1}$ for every index $i$: in fact for every 
$f\in H^i$ and every generator $e\in F^{j}$ we have 
\[(df)e=\debar(fe)-(-1)^if(\debar e)\in IF^{i+j+1}\,.\]

Assume that $H$ is homotopy abelian over $R$, then 
by Lemma~\ref{lem.criterioabelianita} there exists an 
$R$-bilinear map $h\colon H\times H\to H$ of degree $-1$ such that 
\[ [x,y]=dh(x,y)+h(dx,y)+ (-1)^{\deg(x)} h(x,dy),\qquad x,y\in H\,.\]
Consider now the following two elements  $\alpha,\beta\in H$:
\begin{enumerate}
\item  $\alpha\in H^0=\Hom^0_R(F^*,F^*)$ is defined as the identity on 
$F^0$ and $\alpha(F^i)=0$ for $i<0$;

\item  $\beta\in H^n=\Hom^n_R(F^*,F^*)$ extends the above   
$\beta\colon F^{-n}\to F^0$ in the unique possible way, i.e., $\beta(F^i)=0$ for $i>-n$.
\end{enumerate}
Then 
\[[\alpha,\beta]=\alpha \beta-\beta\alpha=\beta,\qquad d\beta=0,\qquad d\alpha\in IH^1\,.\]
By the $R$-bilinearity of $h$ we have 
\[ \beta=[\alpha,\beta]=dh(\alpha,\beta)+h(d\alpha,\beta)+h(\alpha,d\beta)\in IH^n\]
and therefore $\beta(F^{-n})$ is contained in the submodule $IF^0$, which is a contradiction with the choice of $\beta$.
\end{proof}

\begin{corollary}\label{cor.nonabelianitykoszul} 
Let $R$ be a commutative unitary ring and let $K^*$ be the Koszul complex of a sequence $f_1,\ldots,f_n\in R$, $n>0$:
\[ K^*\colon \qquad 0\to K^{-n}\to\cdots\to K^{-1}\to K^0,\qquad K^{-i}={\bigwedge}^iR^n\;.\]
If the ideal $(f_1,\ldots,f_n)\subset R$ is proper, then the DG-Lie algebra $H=\Hom^*_R(K^*,K^*)$ is not homotopy abelian over $R$.
\end{corollary}

\begin{proof} Immediate consequence of Corollary~\ref{cor.nonabelianityresolution}.
Notice however that $\Hom^*_R(K^*,K^*)$ may be homotopy abelian over some subring $S\subset R$ 
even if $(f_1,\ldots,f_n)$ is a proper ideal.

\end{proof}

\bigskip
\section{The endomorphism algebra of Koszul complexes}
\label{sec.endokoszul}

We have already noticed that, except in trivial cases, the DG-Lie algebra $L$ of endomorphisms of the Koszul complex 
of a sequence $f_1,\ldots,f_n\in R$ is not homotopy abelian. In this section we shall prove that if 
$f_1,\ldots,f_n$ is a regular sequence, then $L$ is numerically homotopy abelian.

For every $R$-module $E$ we shall denote by  $E^\vee=\Hom_R(E,R)$ its dual module.
Recall that the  internal product:  
\[ {\bigwedge}^* E^\vee\times {\bigwedge}^* E\xrightarrow{\;\contr\;} {\bigwedge}^{*} E,\]
is the $R$-bilinear map defined recursively by the formulas: 
\[ \begin{split}
f&\contr (v_1\wedge \cdots\wedge v_b)=\sum_{i=1}^b (-1)^{i-1}f(v_i) 
\, v_1\wedge \cdots\wedge\widehat{v_i}\wedge\cdots\wedge v_b\,,\\
(f_1\wedge\cdots\wedge f_a)&\contr (v_1\wedge \cdots\wedge v_b)=
(f_1\wedge\cdots\wedge f_{a-1})\contr(f_a\contr (v_1\wedge \cdots\wedge v_b))\,.\end{split}\]
Let us introduce the  $R$-linear contraction operators: 
\[ \bi_\psi\colon {\bigwedge}^*E\to {\bigwedge}^*E,\qquad \bi_{\psi}(\omega)=\psi\contr \omega,\qquad \psi\in {\bigwedge}^*E^\vee\,;\]
it is immediate to see from the definition above that $\bi_{\psi\wedge\eta}=\bi_\psi\bi_{\eta}$ for every $\psi,\eta\in 
{\bigwedge}^*E^\vee$.

\begin{lemma}\label{lem.kos1} 
In the above setup, consider $\bigwedge^*E$ as a graded $R$-algebra, where the elements of 
$\bigwedge^aE$ have degree $-a$. Then:
\begin{enumerate} 

\item for every $\psi\in E^\vee$ the operator $\bi_{\psi}$ is the unique
 $R$-linear derivation of $\bigwedge^*E$ 
of degree $+1$ such that 
$\bi_\psi(v)=\psi(v)$ for every $v\in E$;

\item $[\bi_{\psi},\bi_{\eta}]=0$ for every $\psi,\eta\in \bigwedge^*E^\vee$.
\end{enumerate}
\end{lemma}

\begin{proof} The first item is clear.
For the second, if 
$\psi,\eta\in E^\vee$ then $[\bi_{\psi},\bi_{\eta}]$ is a derivation
of degree $+2$ of 
$\bigwedge^*E$, and therefore $[\bi_{\psi},\bi_{\eta}]$ annihilates a set of  generators of the $R$-algebra $\bigwedge^*E$.
For the general case $\psi,\eta\in \bigwedge^*E^\vee$ the vanishing of $[\bi_{\psi},\bi_{\eta}]$
follows from the relations: 
\begin{enumerate}
\item $\bi_\alpha\bi_\beta=-\bi_\beta\bi_\alpha$ for every $\alpha,\beta\in  E^\vee$;

\item $\bi_{\psi\wedge\eta}=\bi_\psi\bi_{\eta}$ for every 
$\psi,\eta\in \bigwedge^*E^\vee$; 

\item $[\bi_{\psi},\bi_{\eta}\bi_{\gamma}]=
[\bi_{\psi},\bi_{\eta}]\bi_{\gamma}+(-1)^{\deg(\psi)\deg(\eta)}\bi_{\eta}[\bi_{\psi},\bi_{\gamma}]$ for every 
$\psi,\eta,\gamma\in \bigwedge^*E^\vee$.
\end{enumerate}
\end{proof}

From now on we assume that 
$E$ is the free $R$-module generated by  $e_1,\ldots,e_n$. We shall consider ${\bigwedge}^*E^\vee$ as a differential graded commutative algebra, equipped with the trivial differential, and with  the elements in ${\bigwedge}^iE^\vee$ of degree $i$.
For every $f\in E^{\vee}$ we have 
$(\bi_f)^2=\bi_{f\wedge f}=0$ and then, setting $K^{-i}=\bigwedge^{i}E$,  we get a complex of free $R$-modules: 
\[ K^*\colon \qquad 0\to K^{-n}\xrightarrow{\;\bi_f\;}K^{1-n}\xrightarrow{\;\bi_f\;}\cdots
\xrightarrow{\;\bi_f\;}K^{-1}\xrightarrow{\;\bi_f}K^0=R,\,\]
usually called the Koszul complex of the sequence $f_1=f(e_1)$,\ldots, $f_n=f(e_n)$.

\begin{theorem}\label{thm.nhakoszul}  In the above notation,  
the contraction map
\[ \bi\colon {\bigwedge}^*E^\vee\to \Hom^*_R(K^*,K^*)\]
is a morphism of differential graded algebras. 
If $K^*$  is exact then $\bi$  is surjective in cohomology and therefore the DG-Lie algebra 
$\Hom^*_R(K^*,K^*)$ is numerically homotopy abelian.
\end{theorem}

\begin{proof} The fact that $\bi$ is a morphism of differential graded algebras follows immediately from the formulas $\bi_{\alpha\wedge\beta}=\bi_{\alpha}\bi_{\beta}$ and $[\bi_f,\bi_{\alpha}]=0$.
If $K^*$ is exact and $I=\bi_f(K^{-1})\subset R$ is the ideal generated by $f_1,\ldots,f_n$, then 
the composition of the projections $p\colon K^*\to K^0=R$ and $q\colon R\to R/I$ induces a quasi-isomorphism of complexes 
$qp\colon K^*\to R/I$ and therefore, since $K^*$ is cofibrant, 
the map $qp\colon \Hom^*_R(K^*,K^*)\to \Hom^*_R(K^*,R/I)$ is a surjective quasi-isomorphism. 
Since $\Hom^*_R(K^*,R/I)$ has trivial differential, to conclude the proof  
it is now sufficient to show that the morphism: 
\[ pq\bi\colon {\bigwedge}^*E^\vee\to  \Hom^*_R(K^*,R/I)\]
is surjective. 
This is clear  since 
$q\bi\colon {\bigwedge}^*E^\vee\to  \Hom^*_R(K^*,R)$ is an isomorphism of graded $R$-modules and 
$p\colon \Hom^*_R(K^*,R)\to \Hom^*_R(K^*,R/I)$ is surjective.
\end{proof}

Let now $p\le 0$ be a fixed integer  and denote by 
\[ K^{< p}\colon \qquad 0\to K^{-n}\xrightarrow{\;\bi_f\;}K^{1-n}\xrightarrow{\;\bi_f\;}\cdots
\xrightarrow{\;\bi_f\;}K^{p-1}\]
the truncation of the Koszul complex in degrees $<p$, 
by $j\colon K^{<p}\to K^*$ 
the inclusion and by
$\pi\colon K^*\to K^{<p}$ the projection. 
Notice that $\pi$ is a morphism of complexes, while $j$ is
not a morphism of complexes.

\begin{lemma} For every $\psi,\eta\in \bigwedge^*E^\vee$ we have:
\begin{enumerate}
\item $\pi\bi_\psi=\pi\bi_\psi j\pi\in \Hom^*_R(K^*,K^{<p})$; 

\item $[\pi\bi_\psi j,\pi\bi_{\eta}j]=\pi[\bi_\psi,\bi_{\eta}]j\in \Hom^*_R(K^{<p},K^{<p})$; 

\item the differential in 
$\Hom^*_R(K^{<p},K^{<p})$ is $[\pi\bi_f j,-]$.
\end{enumerate}

In particular,  
\[\bi^{<p} \colon {\bigwedge}^*E^\vee\to \Hom^*_R(K^{<p},K^{<p}),\qquad 
\bi^{<p}_\psi=\pi\bi_{\psi}j\,,\] 
is a morphism of DG-Lie  algebras.
\end{lemma}

\begin{proof} For every $x\in K^*$ we have $x-j\pi(x)\in K^{\ge p}=\ker \pi$ 
and $\bi_\psi( K^{\ge p})\subset K^{\ge p}$. This implies that $\pi\bi_\psi=\pi\bi_\psi j\pi$. 
The second item follows from the equalities
\[ \pi\bi_\psi j\pi\bi_{\eta}j=\pi\bi_\psi\bi_{\eta}j,\qquad 
\pi\bi_\eta j\pi\bi_{\psi}j=\pi\bi_\eta\bi_{\psi}j\;,\]
and the third from the fact that $\pi\bi_f j$ is the differential in $K^{<p}$.
The fact that $\bi^{<p}$ is a morphism of DG-Lie algebra follows trivially from the above items.
\end{proof}

\begin{theorem}\label{thm.nhakoszultronco} 
If the Koszul complex $K^*$  is exact, then the morphism of DG-Lie algebras 
\[ \bi^{<0}\colon \bigwedge^{*}E^\vee\to \Hom^*_R(K^{<0},K^{<0})\]
is surjective in cohomology, and then $\Hom^*_R(K^{<0},K^{<0})$ is numerically homotopy abelian.
\end{theorem}

\begin{proof} Denote as above by $I\subset R$ the ideal generated by $(f_1,\ldots,f_n)$, then the map 
$\bi_f\colon K^{-1}\to I$
gives a quasi-isomorphism of complexes $p\colon K^{<0}\to I[1]$ and then, since $K^{<0}$ is cofibrant
the induced map 
\begin{equation}\label{equ.qik}
\Hom^*_R(K^{<0},K^{<0})\to \Hom^*_R(K^{<0},I[1])\end{equation}
is a quasi-isomorphism. In particular the cohomology of 
$\Hom^*_R(K^{<0},K^{<0})$ is concentrated in degrees $0,\ldots,n-1$.
Since 
\[\bi^{<0}\colon {\bigwedge}^{n-1}E^\vee\to \Hom_R({\bigwedge}^{ n}E,E)=\Hom^{n-1}_R(K^{<0},K^{<0}),\]
is an isomorphism and $\Hom^*_R(K^{<0},K^{<0}))$ is concentrated in degrees $\le n-1$, 
the map ${\bigwedge}^{n-1}E^\vee\to 
H^{n-1}(\Hom^*_R(K^{<0},K^{<0}))$ is surjective.
 
Let now $0\le m<n-1$ be a fixed integer, and let  
$\alpha\in \Hom^m_R(K^{<0},K^{<0})$ be a cocycle, and  
consider its components 
\[ \alpha_{-m-1-i}\colon K^{-m-1-i}\to K^{-1-i},\qquad i=0,\ldots,n-1-m,\]
together with  the morphism 
\[\beta=\bi_f\alpha_{-m-1}\colon K^{-m-1}\to R\;.\]

Since $\alpha$ is a cocycle, $\alpha_{-m-1}\bi_f=\pm \bi_f\alpha_{-m-2}$ and then 
$\beta \bi_f=\pm (\bi_f)^2\alpha_{-m-2}=0$. Equivalently 
$\beta$ is a cocycle  of degree $m+1$ in the dual Koszul complex $\Hom_R^*(K^*,R)$.
Since $m+1<n$, by the self-duality of the Koszul complex \cite[Prop. 17.15]{eisenbud}, the cocycle 
$\beta$ is also a coboundary and there exists 
$\gamma\colon \bigwedge^{m}E\to R$ such that $\beta=(-1)^m\gamma \bi_f$.
In view of the isomorphism $\bigwedge^{m}E^\vee\simeq (\bigwedge^{m}E)^\vee$ there exists $\eta\in 
\bigwedge^{m}E^\vee$ such that 
$\bi_{\eta}(x)=\gamma(x)$ for every $x\in K^{-m}$.

For every 
$x\in K^{-m-1}$ we have $\beta(x)=(-1)^m\bi_{\eta}\bi_f(x)$ and then
\[ \bi_f(\alpha-\bi_{\eta})(x)=\beta(x)-\bi_f\bi_{\eta}(x)=\beta(x)-(-1)^m\bi_{\eta}\bi_f(x)=0\,.\]

We have therefore proved that $\alpha$ and $\bi_{\gamma}$ have the same image in 
$\Hom^*_R(K^{<0},I[1])$, and the conclusion follows by \eqref{equ.qik}.
\end{proof}

\bigskip
\section{Deformations of coherent sheaves via locally free resolutions}

In this section we briefly recall the construction of the DG-Lie algebra controlling the deformation theory of a coherent sheaf given in \cite{FIM,meazzini}.
Here and in the following we denote by $\K$ a fixed field of characteristic 0, by $\Set$, $\Grpd$ and $\Art_{\K}$ the categories of 
Sets, Groupoids and local Artin $\K$-algebras  with residue
field $\mathbb K$, respectively. For every $A\in \Art_{\K}$ we shall denote by $\mathfrak{m}_A$ its maximal ideal.

\subsection{Deligne groupoids} For every DG-Lie algebra $L=\oplus_{i\in\mathbb Z} L^i$ over $\K$, we denote by $\sC_L\colon \Art_\K\to\Grpd$ the action groupoid of the gauge action 
\[ \exp(L^0\otimes\mathfrak{m}_A)\times \MC_L(A)\xrightarrow{\;*\;}\MC_L(A),\]
where 
\[A\in \Art_{\K},\qquad \MC_L(A)=\{x\in L^1\otimes\mathfrak{m}_A\mid dx+\frac{1}{2}[x,x]=0\}\]
is the set of Maurer-Cartan element and the gauge action may be defined by the formula, see e.g. \cite{GoMil1,ManRendiconti,ManettiSeattle}:
\[
e^a \ast x:=x+\sum_{n\geq 0} \frac{ [a,-]^n}{(n+1)!}([a,x]-da)\,.
\]
We also denote by $\Def_L=\pi_0(\sC_L)\colon \Art_\K\to\Set$ the deformation functor associated to $L$, namely:
\[\Def_L(A)=\frac{\operatorname{MC}(L\otimes\mathfrak m_A)}{\;\operatorname{exp}(L^0\otimes\mathfrak m_A)\;},\qquad 
A\in \Art_{\K}\,.\]
Given two objects $x,y$ of $\sC_{L}(A)$, the morphisms between them are 
\[ \Mor_{\sC_{L}}(x,y)=\{e^a\in \exp(L^0\otimes \mathfrak{m}_A)\mid e^a\ast x=y\}.\]
The irrelevant stabilizer of an element $x\in\MC_L(A)$  is the subgroup of 
$I(x)\subseteq \operatorname{Mor}_{\sC_L}(x,x)$ defined in the following way:  
\[I(x)=\left\{e^{du+[x,u]}\;\middle|\;u\in\, L^{-1}\otimes \mathfrak{m}_A\right\}.\]
Then the the Deligne groupoid $\Del_L\colon \Art_{\K}\to \Grpd$ is defined by taking the same objects as $\sC_L$ and by morphisms
\[\operatorname{Mor}_{\Del_L}(x,y)=\frac{\operatorname{Mor}_{\sC_L}(x,y)}{I(x)}\cong
\frac{\operatorname{Mor}_{\sC_L}(x,y)}{I(y)}\,.\]
We refer to \cite{FIM,semireg2011} for full details and for the proof that $\Del_L$ is properly defined. 
The above functors are homotopy invariant in the sense described by  the following theorem:

\begin{theorem}\label{homotopic_invariance}
Let $f\colon L\to M$ be a quasi-isomorphism of differential graded Lie algebras. Then
\begin{enumerate}
\item The induced morphism on deformation functors $f\colon\Def_L\to\Def_M$
is an isomorphism.
\item The induced morphism $\Del_L\to\Del_L$ is an equivalence 
of groupoids. If furthermore $L$ and $M$ are positively graded, the equivalence holds
at the level of action groupoids $\sC_L\to\sC_M$.
\end{enumerate}
\end{theorem}

\begin{proof} This is nowadays well known: different proofs can be found for instance in 
\cite{GoMil1,Konts,ManettiDGLA,ManRendiconti}.
\end{proof}

\subsection{Totalization of semicosimplicial DG-Lie algebras}
Let  
\[L\colon\qquad\xymatrix{L_0\ar@<0.5ex>[r]\ar@<-0.5ex>[r] &L_1\ar@<1ex>[r]\ar@<0ex>[r]\ar@<-1ex>[r] &L_2\ar@<1.5ex>[r]\ar@<0.5ex>[r]\ar@<-0.5ex>[r]\ar@<-1.5ex>[r]&\dots}\]
be a semicosimplicial DG-Lie algebra: this means that every $L_i$ is a DG-Lie algebra and the arrows in the diagram are the morphisms of DG-Lie algebras
\[ \delta_k\colon L_{n-1}\to L_n,\qquad k=0,\ldots,n\,,\]
subjected to the semicosimplicial identities $\delta_l\delta_k=\delta_{k+1}\delta_l$, for any $l\leq k$.

The cochain complex $C(L)$ associated to $L$ is defined as the differential graded vector space
\[C(L)=\left(\prod_{n\geq 0}{L_n[-n]},\; d+\delta\right)\,.\]
More precisely, for every degree $p\in \Z$ we have $C(L)^p=\prod_{n\geq0}{L_n^{p-n}}$, and the differential 
$d+\delta\colon C(L)^p\to C(L)^{p+1}$ is the sum of:
 \begin{itemize}
 \item $d=\sum_{n\geq 0}{(-1)^n d_n}$, where  $d_n$ is the differential of $L_n$;
 \item $\delta=\sum_{k\geq 0}{(-1)^k \delta_k}$.
 \end{itemize}

For every integer $n\ge 0$, denote by 
\[\Omega_n=\frac{\mathbb K[t_0,\dots,t_n,dt_0,\dots dt_n]}{(1-\sum{t_i},\sum{dt_i})}\]
the differential graded algebra of polynomial differential forms on the standard simplex of dimension $n$. 
Then the collection $\Omega_{\bullet}=\{\Omega_n\}_{n\ge 0}$ has a natural structure of  simplicial differential graded algebra, and the face operators
\[\delta_k^*\colon  \Omega_n\to \Omega_{n-1},\qquad k=0,\ldots,n\,,\]
are the morphisms of differential graded algebras such that 
\[ \delta_k^*(t_i)=\begin{cases} t_i&\text{ if }i< k\\
0& \text{ if }i=k\\
t_{i-1}&\text{ if }i>k\end{cases}\]
 
The (Thom-Whitney-Sullivan) totalization of the  
 semicosimplicial DG-Lie algebra $L$ is the DG-Lie subalgebra of  $\prod_{n\geq 0}\Omega_n\otimes L_n$ defined by 
 \[\operatorname{Tot}(L)=\left\{(x_n)\in\prod_{n\geq 0}{\Omega_n\otimes L_n}\; \middle|\; 
(\delta_k^*\otimes Id)x_n=(Id\otimes\delta_k)x_{n-1}
\;\text{for every }0\leq k\leq n\right\}\]
 
By a well known theorem by Whitney, see e.g. \cite{getzler04,lunardon,navarro} and references therein,  there exists a natural pair of quasi-isomorphisms of complexes
$C(L)\xrightarrow{E}\operatorname{Tot}(L)\xrightarrow{I}C(L)$ such that $IE=Id$. In particular, the  cohomology of the totalization $\operatorname{Tot}(L)$ is the same  of the cohomology of the cochain complex $C(L)$.  

\begin{remark}\label{rem.spectralsequence}
For later use we point out that the composition of $I\colon \operatorname{Tot}(L)\to C(L)$ with the projection 
$C(L)\to L_0$ is a morphism of DG-Lie algebras, since it is the natural projection 
$\operatorname{Tot}(L)\to L_0$. The complex $C(L)$ admits the complete and exhaustive filtration $F^p=\prod_{n\geq p}{L_n[-n]}$ whose the first page in the associated spectral sequence is $E_1^{p,q}=H^p(L_p)$. Moreover the natural map 
\[ H^n(\operatorname{Tot}(L))=H^n(C(L))\twoheadrightarrow E_{\infty}^{0,n}\hookrightarrow E_{1}^{0,n}=H^n(L_0)\]
coincides with the map induced in cohomology by the projection $\operatorname{Tot}(L)\to L_0$.  
\end{remark}

\subsection{Descent of Deligne groupoids}
Let  
\[\xymatrix{G\colon\qquad &G_0\ar@<0.5ex>[r]\ar@<-0.5ex>[r] &G_1\ar@<1ex>[r]\ar@<0ex>[r]\ar@<-1ex>[r] &G_2\ar@<1.5ex>[r]\ar@<0.5ex>[r]\ar@<-0.5ex>[r]\ar@<-1.5ex>[r]&\dots}\]
be a semicosimplicial groupoid, i.e., a semicosimplicial object in the category $\Grpd$.
Denoting as above by $\delta_k$ the face operators in $G$, by semicosimplicial identities we get  in particular that     \[\delta_0\delta_0=\delta_1\delta_0,\qquad  \delta_0\delta_1=\delta_2\delta_0,\qquad  \delta_1\delta_1=\delta_2\delta_1\,.\]
 Mimicking the construction of nonabelian 1-cocycles 
we define $Z^1(G)$ as the set of 
pairs 
\[(l,m)\in \operatorname{Obj}(G_0)\times\operatorname{Mor}(G_1)\]
such that  $m\colon\delta_0 l\to\delta_1 l$ and the  \emph{cocycle diagram} 
\begin{equation}\label{equ.cocyclediagram}
\begin{matrix}\xymatrix{&\delta_0\delta_0l\ar@{=}[dl]\ar[dr]^{\delta_0m}&\\
\delta_1\delta_0l\ar[d]^{\delta_1m}\;\;&&\;\;\delta_0\delta_1l\ar@{=}[d]\\
\delta_1\delta_1l\ar@{=}[dr]\;\;&&\;\;\delta_2\delta_0l\ar[dl]_{\delta_2m}\\
&\delta_2\delta_1l&}\end{matrix}
\end{equation}
is commutative in $G_2$. The total descent 
groupoid $\operatorname{Tot}(G)$ is defined in the following way:
\begin{itemize}

\item  the set of objects of  $\operatorname{Tot}(G)$ is  $Z^1(G)$;

\item the morphisms between $(l_0,m_0),$ and $(l_1,m_1)$ are morphisms $a\in \operatorname{Mor}(G_0)$ between $l_0$ and $l_1$ making the following diagram:
 \[\xymatrix{\delta_0 l\ar[r]^{m_0}\ar[d]^{\delta_0 a} &\delta_1 l_0\ar[d]^{\delta_1 a}\\
 		\delta_0 l_1\ar[r]^{m_1} &\delta_1l_1}\]
 		 commutative in $G_1$.	
\end{itemize}

\begin{theorem}\cite[Theorem 7.6]{FIM}\label{thm.descent}
Let  
\[L\colon\qquad\xymatrix{L_0\ar@<0.5ex>[r]\ar@<-0.5ex>[r] &L_1\ar@<1ex>[r]\ar@<0ex>[r]\ar@<-1ex>[r] &L_2\ar@<1.5ex>[r]\ar@<0.5ex>[r]\ar@<-0.5ex>[r]\ar@<-1.5ex>[r]&\dots}\]
be a semicosimplicial DG-Lie algebra such that $H^j(L_i)=0$ for every $i$ and every $j<0$. Consider the associated totalization $\operatorname{Tot}(L)$ and the corresponding  semicosimplicial Deligne groupoid:
  \[\xymatrix{\Del_L\colon &\Del_{L_0}\ar@<0.5ex>[r]\ar@<-0.5ex>[r] &\Del_{L_1}\ar@<1ex>[r]\ar@<0ex>[r]\ar@<-1ex>[r] &\Del_{L_2}\ar@<1.5ex>[r]\ar@<0.5ex>[r]\ar@<-0.5ex>[r]\ar@<-1.5ex>[r]&\dots}\,.\]
Then there exists a natural isomorphism of functors $\Art_{\K}\to\Set$
  \[\Def_{\operatorname{Tot}(L)}\cong\pi_0\operatorname{Tot}(\Del(L))\,.\]
 \end{theorem}
A result of this type appears already in \cite[Theorem 4.1]{hinichDDG};
here under the stronger hypothesis $N^i=0$ for $i<0$ Hinich proves the
equivalence at the level of Deligne groupoids.
  
\subsection{Deformations of coherent sheaves via locally free resolutions}

Let $X$ be a separated noetherian scheme over a field $\K$ of characteristic 0 and $\sF$ a coherent sheaf on $X$. 
Every locally free resolution $\mathcal E^{\bullet}\to \mathcal F$ gives a sheaf 
$\mathcal Hom^*_{\Oh_X}(\mathcal{E}^{\bullet},\mathcal{E}^{\bullet})$ of DG-Lie algebras over the sheaf of commutative rings $\Oh_X$.
Let now $\mathcal U=\left\{U_i\right\}_{i\in I}$ be an affine open cover of $X$, then the \v{C}ech cochains  
of $\mathcal Hom^*_{\Oh_X}(\mathcal{E}^{\bullet},\mathcal{E}^{\bullet})$ in the cover $\sU$, i.e., 
\[\xymatrix{\prod_i{\End^*_{\Oh_{U_i}}(\mathcal{E}^{\bullet}|_{U_i})}\ar@<0.5ex>[r]\ar@<-0.5ex>[r] &\prod_{i,j}{\End^*_{\Oh_{U_{ij}}}(\mathcal{E}^{\bullet}|_{U_{ij}})}\ar@<1ex>[r]\ar@<0ex>[r]\ar@<-1ex>[r] &
\prod_{i,j,k}{\End ^*_{\Oh_{U_{ijk}}}(\mathcal{E}^{\bullet}|_{U_{ijk}})}\ar@<1.5ex>[r]\ar@<0.5ex>[r]\ar@<-0.5ex>[r]\ar@<-1.5ex>[r]&\dots}\]
has a natural structure of  semicosimplicial differential graded Lie algebra whose cochain complex is the one computing \v{C}ech hypercohomology. For simplicity of notation denote this semicosimplicial DG-Lie algebra by 
$L(\sU,\sE^{\bullet})$.
Notice that for every multiindex $\alpha$ in the nerve of the covering, the DG-Lie algebra 
$\End^*_{\Oh_{U_\alpha}}(\mathcal{E}^{\bullet}|_{U_\alpha})$ has trivial cohomology in negative degree.
In fact, since $U_{\alpha}$ is affine, we have 
\[ \End^*_{\Oh_{U_\alpha}}(\mathcal{E}^{\bullet}|_{U_\alpha})=\Hom^*_{\Oh_X(U_\alpha)}(
\mathcal{E}^{\bullet}(U_\alpha),\mathcal{E}^{\bullet}(U_\alpha))\]
and $\mathcal{E}^{\bullet}(U_\alpha)\to \sF(U_\alpha)$ is a projective resolution. Therefore for every $i\in \Z$ we have
\[H^{i}(\Hom^*_{\Oh_X(U_\alpha)}(\mathcal{E}^{\bullet}(U_\alpha),\mathcal{E}^{\bullet}(U_\alpha))\simeq
\Ext^i_{\Oh_X(U_\alpha)}(\sF(U_\alpha),\sF(U_\alpha))\,.\]

In particular Theorem~\ref{thm.descent} applies to the semicosimplicial DG-Lie algebra $L(\sU,\sE^{\bullet})$: as a consequence, Fiorenza, Iacono and Martinengo proved  the following result: 
 
\begin{theorem}[{\cite[Section 5]{FIM}}]\label{thm.descentesheaf}
In the above notation the quasi-isomorphism class of the DG-Lie algebra $\operatorname{Tot}(L(\sU,\sE^{\bullet}))$ depends only on $\sF$, i.e., it is independent of the choice of the resolution and of the affine cover. 
The functor $\Def_{\operatorname{Tot}(L(\sU,\sE^{\bullet}))}$ is isomorphic
to the functor $\Def_{\sF}$ of (infinitesimal) deformation of   the coherent sheaf $\sF$.
 \end{theorem}

Although technically complicated, the underlying idea of proof of Theorem~\ref{thm.descentesheaf} is quite easy and dates back to \cite{DtS,Schechtmanarbeit}, cf. also \cite{H04,H05}. 
First, by standard argument of deformation theory, essentially contained mutatis mutandis in \cite{Artinbook,Sernesi}, one prove that for every
$\alpha$ in the nerve of the covering, the Deligne groupoid of the DG-Lie algebra $\End^*_{\Oh_{U_\alpha}}(\mathcal{E}^{\bullet}|_{U_\alpha})$ is isomorphic to the deformation groupoid of infinitesimal deformations of the coherent sheaf $\sF$ over the affine scheme $U_{\alpha}$, see also \cite{carocci} for a fully detailed proof. Then 
the total groupoid of the semicosimplicial Deligne groupoid associated to $L(\sU,\sE^{\bullet})$ compute the deformations of $\sF$ over $X$, since the commutativity of \eqref{equ.cocyclediagram} can be interpreted as the 
usual cocycle conditions to glueing the local deformations of $\sF$ into a global deformation.

Assume for simplicity that $\sU$ is a finite cover and consider its disjoint union $Y=\coprod U_i$ together the open covering map $u\colon Y\to X$. Then the DG-Lie algebra $\prod_i{\End^*_{\Oh_{U_i}}(\mathcal{E}^{\bullet}|_{U_i})}$ controls the deformations of $u^*\sF$ and the natural transformation of deformation functors
$\Def_{\sF}\to \Def_{u^*\sF}$ is induced by the projection
\[ \operatorname{Tot}(L(\sU,\sE^{\bullet}))\to \prod_i{\End^*_{\Oh_{U_i}}(\mathcal{E}^{\bullet}|_{U_i})}\,.\]

\bigskip
\section{Deformations of locally unobstructed coherent sheaves}\label{sec.deformations}

A coherent sheaf $\sF$ on a scheme $X$ is said to be locally unobstructed if there exists an open affine cover 
$\sU=\{U_i\}$ such that every restriction $\sF_{|U_i}$ has unobstructed deformations.%
\footnote{The notion of obstruction to deformations  of  $\sF_{|U_i}$ makes sense even if 
the space of first order deformations is not finite dimensional, and therefore the corresponding functor of Artin rings does not satisfy Schlessinger's conditions, see \cite{FM1,ManettiSeattle}.}   

The construction of the previous section allows to give a proof in the framework of differential graded Lie algebras of the following result about the vanishing of the obstructions of a locally unobstructed coherent sheaf under the local-to-global Ext spectral sequence.

\begin{theorem}\label{thm.vanishingobstructions} 
Let $\sF$ be a locally unobstructed coherent sheaf on a projective scheme $X$ over a field of characteristic 0. Then the obstructions to deforming $\sF$ are contained in the kernel of the natural map
\[ \alpha_0\colon \Ext^2_{X}(\sF,\sF)\to E_2^{0,2}=H^0(X,\EXT^2_{X}(\sF,\sF))\,.\]   
In particular, if $\alpha_0$ is injective then $\sF$ has unobstructed deformations.
\end{theorem}

\begin{proof} Let $\sU=\{U_i\}$ be a open affine cover such that every restriction $\sF_{|U_i}$ has unobstructed deformations. Fix a locally free resolution $\sE^{\bullet}\to \sF$  and let $L(\sU,\sE^{\bullet})$ be the semicosimplicial DG-Lie algebra as in Theorem~\ref{thm.descentesheaf}. 
According to Remark~\ref{rem.spectralsequence} the pages  $E_r^{p,q}$, $r\ge 2$, of the spectral sequence may be computed by considering the cochain complex $C(L(\sU,\sE^{\bullet}))$. Since $E_1^{p,q}=0$ for every 
$p<0$ we have 
\[ H^0(X,\EXT^n_{X}(\sF,\sF))=E_2^{0,n}\subseteq E_1^{0,n}=\prod_iH^0(U_i,\EXT^n_{U_i}(\sF,\sF))\,,\]
and then it is sufficient to prove that every obstruction is annihilated by the map 
\[ \Ext^2_{X}(\sF,\sF)=H^2(\operatorname{Tot}(L(\sU,\sE^{\bullet})))\to 
\prod_iH^0(U_i,\EXT^2_{U_i}(\sF,\sF)).\]   
induced by the morphism of DG-Lie algebras
\[ \operatorname{Tot}(L(\sU,\sE^{\bullet}))\to L(\sU,\sE^{\bullet})_0=\prod_i\End_{\Oh_{U_i}}(\sE^{\bullet}_{|U_i})\,.\]
Since the DG-Lie algebra $\End_{\Oh_{U_i}}(\sE^{\bullet}_{|U_i})$ controls the infinitesimal deformations 
of $\sF_{|U_i}$ and morphisms of DG-Lie algebras commute with obstruction maps, it follows that obstructions to deformations of $\sF$ are annihilated by every map $\operatorname{Tot}(L(\sU,\sE^{\bullet}))\to 
\End_{\Oh_{U_i}}(\sE^{\bullet}_{|U_i})$.
\end{proof}

The expert reader immediately recognize that the use of DG-Lie algebras in the proof of Theorem~\ref{thm.vanishingobstructions} is not strictly necessary and the same result can be proved by classical methods. 
In fact, for every index $i$ the composite map 
\[ \Ext^2_{X}(\sF,\sF)\to H^0(X,\EXT^2_{X}(\sF,\sF))\to H^0(U_i,\EXT^2_{U_i}(\sF,\sF))\]   
 is the obstruction map associated to the natural transformation of deformation functors
$\Def_{\sF}\to \Def_{\sF|U_i}$ and the functor $\Def_{\sF|U_i}$ is unobstructed by assumption. However, the approach via DG-Lie algebras not only it is interesting on its on, but also provide an immediate generalization of the theorem to derived deformations, i.e., to deformations of $\sF$ over differential graded local Artin rings. 
Without entering into details of derived deformation theory, we only recall that, over a field of characteristic 0, there exists an equivalence of categories, between the homotopy 
category of DG-Lie algebras and the category of derived deformation problems \cite{Lur,PridUDDT}: when the derived deformation problem is represented by a set-valued functor in the category of DG-Artin rings, the staightforward generalization of Maurer-Cartan equation and gauge equivalence gives the required equivalence, cf. \cite{EDF,PridUDDT}. Moreover unobstructed derived deformation problems corresponds precisely to homotopy abelian DG-Lie algebras 
(this is implicitly proved in \cite[Sec. 7]{EDF} since the minimal $L_{\infty}$ model is reconstructed from derived obstruction maps; for an explicit statement we refer to \cite[Thm. 6.3]{lazarev2012}).    

\subsubsection*{A remark about obstructions to derived deformations}

If $L$ is a DG-Lie algebra the notions of $\MC_L(A)$ and $\Def_L(A)$ extend naturally to the case 
where $A$ is a differential graded local $\K$-algebras with residue field $A\xrightarrow{\pi}\K$:
\[\begin{split} 
\MC_L(A)&=\left\{x\in (L\otimes\mathfrak{m}_A)^1\;\middle|\; dx+\frac{1}{2}[x,x]=0\right\}\,,\\
\Def_L(A)&=\frac{\MC_L(A)}{\exp(L\otimes\mathfrak{m}_A)^0}\,,\end{split}\]
and, as in the classical case, these two functors have the same obstruction theory.
A simple computation, see \cite{EDF} for details, shows that if 
$f\colon A\to B$ is a surjective morphism of DG-Artin rings such that its kernel is an acyclic complex annihilated by the maximal ideal $\mathfrak{m}_A$, then $f\colon \Def_L(A)\to \Def_L(B)$ is a bijective map.

If $t\in A$ is a homogeneous non-trivial closed element of degree $-n$ annihilated by the maximal ideal $\mathfrak{m}_A$, we have a small extension:
\begin{equation}\label{equ.smalldgextension} 
0\to \K[n]\xrightarrow{t}A\xrightarrow{p}A/(t)\to 0\end{equation}
and a ``connecting'' morphism $rq^{-1}\colon \Def_L(A/(t))\to H^{2+n}(L)$ is defined in the following way. Consider a formal symbol $s$ of degree $-n-1$, the differential graded Artin $\K$-algebra 
\[ A\oplus \K s,\qquad \mathfrak{m}_As=s^2=0,\quad ds=t,\] 
and the  surjective morphisms:
\[q\colon A\oplus \K s\to  A/(t),\quad q(a+bs)=p(a),\]
\[r\colon A\oplus \K s\to \K\oplus \K s,\quad 
 r(a+bs)=\pi(a)+bs\,.\]
As observed above, the map 
$q\colon \Def_L(A\oplus \K s)\to \Def_L(A/(t))$ is bijective and it is not difficult to see \cite[pag. 737]{EDF} that an element 
$x\in \Def_L(A/(t))$ lifts to $\Def_L(A)$ if and only if 
\[rq^{-1}(x)=0\in 
\Def_L(\K\oplus \K s)\cong H^{2+n}(L)\,.\]
If $L$ is an  abelian DG-Lie algebra, then 
$\Def_L(A/(t))=H^1(L\otimes\mathfrak{m}_A/{(t)})$ and the map 
$rq^{-1}\colon \Def_L(A/(t))\to H^{2}(L[n])=H^{2+n}(L)$ is the connecting homomorphism of the short exact sequence of complexes
\[ 0\to L[n]\to L\otimes\mathfrak{m}_A\to L\otimes\dfrac{\mathfrak{m}_A}{(t)}\to 0\,.\]

\begin{definition} We shall say that the morphism $rq^{-1}\colon \Def_L(A/(t))\to H^{2+n}(L)$ 
associated to the small extension \eqref{equ.smalldgextension} is an 
\emph{obstruction map} if the morphism $p$ is surjective in cohomology, or equivalently if 
$t$ is not exact in $A$.
\end{definition}

If $L$ is abelian, every obstruction map vanishes by K\"unneth formula. Since the maps $rq^{-1}$ depends only on $\Def_L$, that in turn depends  only on the homotopy class of $L$, it follows that 
if $L$ is homotopy abelian, then every (derived) obstruction map vanishes. 
Conversely, 
 using the $L_{\infty}$-minimal model of $L,$ it is possible to prove that: if all obstructions vanish, then $L$ is homotopy abelian. It is worth to point out, see \cite[Thm. 6.3]{lazarev2012} and 
\cite[Proof of Thm. 7.1]{EDF}, that the homotopy abelianity of $L$ is equivalent to the vanishing of obstructions arising from the subclass of small extensions \eqref{equ.smalldgextension} where $A$ has 
zero differential.

\begin{theorem}\label{thm.vanishingobstructionsderived} Let $\sF$ be a  coherent sheaf on a projective scheme $X$ over a field $\K$ of characteristic 0, and assume that $\sF$ is locally unobstructed in the derived sense. 
Then the obstructions to derived deformations of $\sF$ are contained in the kernel of the natural map of graded vector spaces 
\[ \alpha\colon \bigoplus_n\Ext^{n+2}_{X}(\sF,\sF)\to \bigoplus_n H^0(X,\EXT^{n+2}_{X}(\sF,\sF))\,.\]   
In particular, if $\alpha$ is injective then the DG-Lie algebra $\operatorname{Tot}(L(\sU,\sE^{\bullet}))$ controlling deformations of $\sF$ is homotopy abelian over $\K$.
\end{theorem}

\begin{proof} It the same sitution as in the proof of Theorem~\ref{thm.vanishingobstructions} 
the assumption implies that every DG-Lie algebra  $\End_{\Oh_{U_i}}(\sE^{\bullet}_{|U_i})$ is homotopy abelian over $\K$ and the map of graded vector spaces $\alpha$ is the (derived) obstruction map 
induced by the morphism of DG-Lie algebras
\[ \operatorname{Tot}(L(\sU,\sE^{\bullet}))\to L(\sU,\sE^{\bullet})_0=\prod_i\End_{\Oh_{U_i}}(\sE^{\bullet}_{|U_i})\,.\]
If $\alpha$ is injective, then $\operatorname{Tot}(L(\sU,\sE^{\bullet}))$ is numerically homotopy abelian, and hence homotopy abelian over the field $\K$ by Remark~\ref{rem.nha}.
\end{proof}

\subsection{The case of locally complete intersection ideal sheaves.}
The results of Section~\ref{sec.endokoszul} imply in particular that the 
assumption about local unobstructness of $\sF$ made in Theorems~\ref{thm.vanishingobstructions} and \ref{thm.vanishingobstructionsderived} is valid when $\sF$ is a 
locally complete intersection  ideal sheaf $\sI\subset \Oh_X$: this means that $\sI$ is a coherent sheaf and 
there exists an affine open cover $\sU=\{U_i\}$ such that on every $U_i$ the quotient sheaf $\Oh_Z=\Oh_X/\sI$ admits a Koszul resolution $\sK^*\to \Oh_Z$, where
\[ \sK^*\colon \quad 0\to \wedge^r\Oh^r_{U_i}\to \cdots \to \wedge^2\Oh^r_{U_i}\to \Oh^r_{U_i}\xrightarrow{(f_1,\dots,f_r)}\Oh_{U_i}\,,\]
\[ f_1,\ldots,f_r\in \sI(U_i),\qquad H^0(\sK^*)=\frac{\Oh_{U_i}}{\;\sI_{|U_i}}\,,\qquad H^i(\sK^*)=0\;\text{ for every }i\not=0\,.\]
We do not require that the integer $r$ is the same for every $U_i$: for instance if $U_i$ does not intersect the closed subscheme defined by $\sI$ we have $r=1$ and $f_1$ an invertible element.

\begin{lemma} Let $\sI\subset \Oh_X$ be a locally complete intersection ideal sheaf on a projective scheme over a field $\K$ of characteristic 0, and denote by $\Oh_Z=\Oh_X/\sI$ the structure sheaf on the closed subscheme $Z$ defined by $\sI$.

Let $\sF$ be either a line bundle over $Z$ or $\sF=\sI\otimes \sL$ for a line bundle $\sL$ on $X$. Then for every open affine subset $U\subset X$ the sheaf $\sF_{|U}$ has unobstructed derived deformations.
\end{lemma}

\begin{proof} Since $\Ext^n_U(\sF,\sF)=H^0(U,\EXT^n_U(\sF,\sF))$, by Theorem~\ref{thm.vanishingobstructionsderived} it is sufficient to prove that $\sF_{|U}$ is locally unobstructed in the derived sense, i.e., that 
every point belongs to an open affine set $V\subset U$ such that  the DG-Lie algebra controlling deformations of 
$\sF_{|V}$ is homotopy abelian. Choose $V$ sufficiently small such that $\sI$ is generated by a regular sequence
and either $\sF_{|U}\simeq \Oh_{Z\cap V}$ or $\sF_{|U}\simeq \sI_{V}$.

Since the homotopy class of the DG-Lie algebra controlling deformations is independent of the choice of the resolution we can consider the Koszul resolutions
\[ \sK^*\xrightarrow{\text{q-iso}}\Oh_{Z\cap V},\qquad \sK^{<0}\xrightarrow{\text{q-iso}}\sI_{|V},\]
where
\[ \sK^*\colon \quad 0\to \wedge^r\Oh^r_{V}\to \cdots \to \wedge^2\Oh^r_{V}\to \Oh^r_{V}\xrightarrow{(f_1,\dots,f_r)}\Oh_{V}\,.\]
According to Theorems~\ref{thm.nhakoszul} and~\ref{thm.nhakoszultronco} both the DG-Lie algebras $\End_{\Oh_V}^*(\sK^*)$ and 
$\End_{\Oh_V}^*(\sK^{<0})$ are numerically homotopy abelian, and hence homotopy abelian over the field $\K$.
\end{proof}

\begin{corollary}\label{cor.vanishingobstructionsderived} 
Let $\sI\subset \Oh_X$ be a locally complete intersection ideal sheaf on a projective scheme over a field $\K$ of characteristic 0, and denote by $\Oh_Z=\Oh_X/\sI$ the structure sheaf on the closed subscheme $Z$ defined by $\sI$.
Let $\sF$ be either a line bundle over $Z$ or $\sF=\sI\otimes \sL$ for a line bundle $\sL$ on $X$.
Then:\begin{enumerate}

\item  the obstructions to deforming $\sF$ are contained in the kernel of the natural map
\[ \alpha_0\colon \Ext^2_{X}(\sF,\sF)\to H^0(X,\EXT^2_{X}(\sF,\sF))\,.\]   
In particular, if $\alpha_0$ is injective then $\sF$ has unobstructed deformations;

\item the obstructions to derived deformations of $\sF$ are contained in the kernel of the natural map of graded vector spaces 
\[ \alpha\colon \bigoplus_n\Ext^{n+2}_{X}(\sF,\sF)\to \bigoplus_n H^0(X,\EXT^{n+2}_{X}(\sF,\sF))\,.\]   
In particular, if $\alpha$ is injective then the DG-Lie algebra $\operatorname{Tot}(L(\sU,\sE^{\bullet}))$ controlling deformations of $\sF$ is homotopy abelian over $\K$.
\end{enumerate}
\end{corollary}

\begin{proof} Immediate from the above results.\end{proof}

\begin{example} Let $Z$ be a locally complete intersection subvariety of a smooth projective manifold $X$, with normal sheaf $\sN_{Z|X}$. Then 
$\EXT^n_{X}(\Oh_Z,\Oh_Z)\simeq \wedge^n\sN_{Z|X}$ and then the obstruction to deformations of the coherent sheaf $\Oh_Z$ are contained in the kernel of the map
\[ \Ext^2_X(\Oh_Z,\Oh_Z)\to H^0(Z,\wedge^2\sN_{Z|X})\,.\]

If $H^i(Z,\wedge^k\sN_{Z|X})=0$ for every $k\ge 0$ and every $i>0$, then the coherent sheaf $\Oh_Z$ is unobstructed in the derived sense: in fact by the local-to-global Ext spectral sequence we get that the map $\alpha$ of
Corollary~\ref{cor.vanishingobstructionsderived} is injective.  
\end{example}

\begin{example} Let $Z$ be a locally complete intersection subvariety of codimension $p\geq 2$ of a smooth projective manifold $X$, with ideal sheaf $\sI$. Then $\EXT^k_{X}(\sI,\sI)=0$ for $k\geq p$ and  
\begin{equation}\label{eq:iso} 
\EXT^k_{X}(\sI,\sI)\cong \EXT^k_{X}(\Oh_Z,\Oh_Z)\cong \wedge^k\sN_{Z|X}\quad\text{for}\quad 0\le  k < p\,.
\end{equation}
The first part is clear since, locally, the coherent sheaf $\sI$ has projective dimension $<p$ as 
$\Oh_X$-module. In view of the natural isomorphism 
$\HOM_X(\Oh_Z,\Oh_Z)\xrightarrow{\simeq}\HOM_X(\Oh_X,\Oh_Z)$, the functor $\HOM_X(-,\Oh_Z)$ applied to the short exact sequence $0\to\sI\xrightarrow{j}\Oh_X\xrightarrow{\pi}\Oh_Z\to 0$, give a sequence of natural isomorphisms 
\[  \EXT^k_{X}(\sI,\Oh_Z)\xrightarrow{\cong}\EXT^{k+1}_{X}(\Oh_Z,\Oh_Z),\qquad k\ge 0\,.\]
Then we observe that for every $i\not=0,p-1$ we have 
$\EXT^i_{X}(\sI,\Oh_X)\cong \EXT^{i+1}_{X}(\Oh_Z,\Oh_X)=0$,
while  
the natural map 
\[\EXT^{p-1}_X(\sI,\Oh_X)\to \EXT^{p-1}(\sI,\Oh_Z) \]
is an isomorphism: this follows easily by a direct computation involving 
the Koszul complex as locally free resolution of $\sI$.
Moreover, since $\EXT^{0}_{X}(\Oh_Z,\Oh_X)=\EXT^{1}_{X}(\Oh_Z,\Oh_X)=0$ we have 
$\HOM_X(\sI,\Oh_X)\simeq \HOM_X(\Oh_X,\Oh_X)\simeq \Oh_X$ and in particular the morphism 
\[\HOM_X(\sI,\Oh_X)\xrightarrow{0} \HOM_X(\sI,\Oh_Z) \]
is trivial.

Therefore, the functor  $\HOM_X(\sI,-)$ applied to the short exact sequence $0\to\sI\xrightarrow{j}\Oh_X\xrightarrow{\pi}\Oh_Z\to 0$ gives a sequence of isomorphisms:
\[ \HOM_X(\sI,\sI)\cong \HOM_X(\sI,\Oh_X)\cong \Oh_X,\qquad 
\EXT^{k-1}_{X}(\sI,\Oh_Z)\xrightarrow{\cong}\EXT^{k}_{X}(\sI,\sI)\qquad 0<k<p\,.\]

In conclusion, we can rephrase the result of Corollary~\ref{cor.vanishingobstructionsderived} as follows:
if $p>2$ then the obstruction to deformations of the coherent sheaf $\sI$ are contained in the kernel of the map
\[ \Ext^2_X(\sI,\sI)\to H^0(Z,\wedge^2\sN_{Z|X})\,.\]

If $H^i(Z,\wedge^k\sN_{Z|X})=0$ for every $0\leq k <p$ and every $i>0$, then the local to global morphism $\alpha$ is injective and the coherent sheaf $\sI$ is unobstructed in the derived sense.

\end{example}

\bigskip
\begin{acknowledgement}
 M.M. wishes to acknowledge the    support by Italian MIUR under PRIN project 2015ZWST2C \lq\lq Moduli spaces and Lie theory\rq\rq. F.C. thanks the second named author who proposed  this project for her master thesis in the far 2014 \cite{carocci}.
F.C. was supported by the Engineering and Physical Sciences Research Council [EP/L015234/1]. 
The EPSRC Centre for Doctoral Training in Geometry and Number Theory (The London School of Geometry and Number Theory), University College London
 \end{acknowledgement}

\end{document}